\documentclass[a4paper,12pt]{amsart}
\usepackage{amsmath,amssymb}
\newtheorem{theorem}{Theorem}

\newtheorem{corollary}[theorem]{Corollary}
\newtheorem{lemma}[theorem]{Lemma}
\newtheorem{definition}[theorem]{Definition}
\newenvironment{proof*}{\vskip 2mm\noindent {}}{\hfill $\Box$ \vskip 2mm}
\def\C{\mathbb C}
\def\D{\mathbb D}
\def\dl{\delta}

\title[\hskip-2mm The Bergman kernel and the Carath\'eodory--Eisenman volume]
{Comparison of the Bergman kernel and the Carath\'eodory--Eisenman volume}

\author{Nikolai Nikolov and Pascal J.~Thomas}

\address{N. Nikolov\\ Institute of Mathematics and Informatics\\Bulgarian Academy
of Sciences\\Acad. G. Bonchev 8, 1113 Sofia, Bulgaria\newline
\indent Faculty of Information Sciences\\
State University of Library Studies and Information Technologies\\
Shipchenski prohod 69A, 1574 Sofia, Bulgaria}\email{nik@math.bas.bg}

\address{P.J. Thomas\\
Institut de Mathématiques de Toulouse; UMR5219 \\
Universit\'e de Toulouse; CNRS \\
UPS, F-31062 Toulouse Cedex 9, France} \email{pascal.thomas@math.univ-toulouse.fr}

\keywords{Bergman kernel, Azukawa metric, Carath\'eodory metric, Carath\'eodory--Eisenman volume}

\subjclass[2010]{32F45}

\thanks{The first named author is partially supported by the Bulgarian National Science Fund,
Ministry of Education and Science of Bulgaria under contract DN 12/2. This paper was started
while his was visiting the Paul Saba\-tier University, Toulouse in November 2018
as a guest professor.}

\begin{document}

\begin{abstract} It is proved that for any domain in $\C^n$ the Carath\'eodo\-ry--Eisenman
volume is comparable with the volume of the indicatrix of the Carath\'eodory metric
up to small/large constants depending only on $n.$ Then the ``multidimensional Suita
conjecture" theorem of B\l ocki and Zwonek implies a comparable relationship between
these volumes and the Bergman kernel.
\end{abstract}

\maketitle

In recent years, the interest in holomorphically invariant objects has
grown from quantities stemming from maps to or from the one-dimensional  disc
to quantities related to the $n$-dimensional ball. The main
focus of interest has been the squeezing function, which measures how big
the one-to-one image of a domain can be while remaining inside the unit ball
(and sending a base point to the origin of the ball).

The Carath\'eodo\-ry--Eisenman
``volume" is a variant on that idea, at the infinitesimal level.
Let $\Bbb D\subset\C$ be the unit disc.
Given $D$ be a domain in $\C^n,$ and $z\in D$,
$$
CE_D(z)=\sup\{|\det F'(z)|^2: F\in\mathcal O(D,\D^n)\}.
$$
We are using the polydisc $\D^n$ for technical reasons.
Replacing it by the unit ball in $\C^n,$
we get the same function up to small/large constants independent of $D$.

Unfortunately, the lack of a higher-dimensional analogue to the Koebe quarter theorem
prevents us from relating our results to the squeezing function, but
the behaviour of $CE_D(z)$ can be related to some basic geometric objects associated
to the domain. We need more definitions.

\begin{definition}
Let $D$ be a domain in $\C^n,$
$z,w\in D,$ and $X\in\C^n.$ The pluricomplex Green function $g_D,$
the Azukawa metric $A_D$ and the Carath\'eodory metric $C_D$
are defined in the following way:
$$
g_D(z,w)=\sup\{u(w): u\in\mbox{PSH}(D), u<0, u(\zeta)<\log||\zeta-z||+C\},
$$
$$
A_D(z;X)=\limsup_{\lambda\to 0}\frac{\exp(g_D(z,z+\lambda
X))}{|\lambda|},
$$
$$
C_D(z;X)=\sup\{|f'(z)X|: f\in\mathcal O(D,\D)\}.
$$
\end{definition}

Let $L_h^2(D)$ be the Bergman space of $D$,
i.e. the Hilbert space of all square-integrable holomorphic functions $f$ on $D$.
Let $K_D$ be the
restriction to the diagonal of the Bergman kernel of $D$. Recall that
$$
K_D(z)=\sup\{|f(z)|^2:f\in L_h^2(D), ||f||_{L^2(D)}\le 1\}.
$$

Denote by $\dl_D(z;X)$ the distance from $z$ to $\partial D$ along the vector $X$:
$$
\delta_D(z;X):= \sup\{r>0: z+\lambda X\in D\mbox{ if }|\lambda|<r\}.
$$
Observe that $\dl^{-1}_D(z;X)$ is $1$-homogeneous in $X$, that is,
$$\dl^{-1}_D(z; \mu X) = |\mu| \dl^{-1}_D(z;X),\quad\mu \in \mathbb C,$$
so it is this quantity that we want to compare to the various
infinitesimal metrics that occur in complex analysis.

Let $I_D(z)$ be the indicatrix of $\dl^{-1}_D(z;\cdot),$ that
is,
$$
I_D(z)=\{X\in\Bbb C^n:\dl^{-1}_D(z;X)<1\}.
$$
Then $z+I_D(z)$ is the maximal
balanced subdomain of $D$ centered at $z.$

Set $IA_D$ and $IC_D$ to be the indicatrices of $A_D$ and $C_D.$
Note that $IC_D(z)$ is a convex set.

\begin{definition}
\label{vol}
Denote $V_D,$ $VA_D$ and $VC_D$ the Euclidean volumes of $I_D,$ $IA_D$ and $IC_D,$
respectively.
\end{definition}

Since $C_D\le A_D\le\dl^{-1}_D,$ then $I_D\subset IA_D\subset IC_D$ and hence
\begin{equation}
\label{comp}
V_D\le VA_D\le VC_D.
\end{equation}

On the other hand, if $G$ is a balanced domain, then $(1,f)_{L_2(G)}=0$ for any
$f\in L_2(G)$ with $f(0)=0.$ Therefore $K_G(0)=(\mbox{Vol}\ G)^{-1}$. In particular,
applying this to $G= z+I_D(z) \subset D$,
$$
K_D(z)\le(V_D(z))^{-1}.
$$

The following opposite inequality is called the multidimensional
Suita conjecture (see \cite[Theorem 7.5]{Blo} and \cite[Theorem 2]{BZ}).

\begin{theorem}\label{S} If $D$ is a pseudoconvex domain in $\Bbb C^n,$
then $$K_D\ge(VA_D)^{-1}.$$
\end{theorem}

\begin{corollary}\label{c1} If $D$ is a domain in $\C^n$ and $A_D\ge c\dl^{-1}_D$
for some $c>0,$ then
$$
c^{-2n}(VA_D)^{-1}\ge(V_D)^{-1}\ge K_D\ge(VA_D)^{-1}.
$$
\end{corollary}

Note that if $D$ is $\Bbb C$-convex, resp. convex, then $A_D\ge C_D\ge c\dl^{-1}_D$ with
$c=1/4$, resp. $c=1/2$ (see e.g. \cite[Proposition 1]{NPZ}, resp. the remark
after this proposition). Thus, Corollary \ref{c1} applies to those cases and we
reobtain  $K_D \le c^{-2n}(VA_D)^{-1}$ as in
\cite[Theorem 5]{BZ}.

The aim of this note is to prove a version of Theorem \ref{S} for $CE_D,$
comparing it to the volume of the Carath\'eodory indicatrix (see Definition
\ref{vol}).

\begin{theorem}\label{CE} Let $D$ be a domain in $\Bbb C^n.$ There are constants
$C_n>c_n>0$ depending only on $n$ such that
$$
C_n(VC_D)^{-1}\ge CE_D\ge c_n(VC_D)^{-1}.
$$

In particular, if $D$ is pseudoconvex, then by Theorem \ref{S} and \eqref{comp},
$$C_n K_D\ge CE_D.$$
\end{theorem}

\begin{corollary}\label{c2} If $D$ is domain in $\C^n$ and $C_D\ge c\dl^{-1}_D$
for some $c>0,$ then
$$C_nK_D\ge CE_D\ge c^{2n}c_nK_D.$$
\end{corollary}

\noindent{\it Proof of Corollaries \ref{c1} and \ref{c2}}. We only have to show that
under the assumption $A_D\ge c\dl^{-1}_D$ (resp. $C_D\ge c\dl^{-1}_D$),
$D$ is pseudoconvex. Suppose it is not. By \cite[Theorem 4.1.25]{Hor}, after an affine change
of coordinates, we may suppose that $0\in\partial D$ and
$$
D\supset D_{c,r}:=\{z\in r\D^2:\mbox{Re} z_1+(\mbox{Re} z_2)^2<c(\mbox{Im} z_2)^2\},\quad c>1,\ r>0.
$$

Recall that the Kobayashi-Royden metric of a domain $G$ in $\C^n$ is given by
$$
\kappa_G(z;X)=\inf\{|\alpha|:\exists\varphi\in\mathcal O(\mathbb
D,G):\varphi(0)=z,\alpha\varphi'(0)=X\}.
$$
It follows by the proof of \cite[Theorem 1.1]{Kra} that
$$
\limsup_{\delta\to 0+}\delta^{3/4}\kappa_{D_{c,r}}((-\delta,0);e_1)<\infty.
$$
Then the inequalities $C_D\le A_D\le \kappa_D\le \kappa_{D_{c,r}}$ lead to a contradiction.

\begin{proof*}{\it Proof of Theorem \ref{CE}}.
Let $a\in D.$ Note that $D_a:=IC_D(a)$ is a
convex balanced domain centered at $a,$ and hence
\begin{equation}
\label{balind}
C_D(a;X)=C_{D_a}(a;X)=\dl^{-1}_{D_a}(a;X).
\end{equation}
The proof of \cite[Proposition 14]{NPZ} rests on the construction,
in a $\mathbb C$-convex domain $D$,
of an orthonormal basis $e_1,\dots,e_n$ of $\C^n$ such that
for $1 \le j \le n$,
$$
\dl_{D}(a;e_j) = \mbox{dist}\left(a, a+\mbox{Span}(e_k, k\ge j) \setminus D \right).
$$
Since $D_a$ is convex, using \eqref{balind}, we deduce from \cite[(4)]{NPZ}  that
one may find a constant $k_n>0$ depending only on $n$
such that
\begin{equation}\label{1}
k_n\sum_{j=1}^n|X_j|r_j(a)\le C_D(a;X)\le\sum_{j=1}^n|X_j|r_j(a),
\end{equation}
where the $X_j$ are the coordinates of $X$ in the basis $e_1,\dots,e_n$ and
$r_j(a)=C_D(a;e_j).$

Let $\Pi_D(a)=\prod_{j=1}^n r_j(a)$.

\begin{lemma}
There exists a map $F=(f_1,\dots,f_n) \in \mathcal O(D,\D)$ such that
$$|\det F'(a)| \ge (k_n)^n\Pi_D(a).$$
\end{lemma}
\begin{proof}
Let $f_1 \in \mathcal O(D,\D)$ be extremal for the Carath\'eodory metric in the $e_1$
direction, thus $\left| \frac{\partial f_1}{\partial z_1}(a)\right|=|f_1'(a)e_1|= r_1(a)$.
We proceed recursively: suppose we already have chosen
$f_i \in \mathcal O(D,\D)$, $1\le i \le m$, such that
$$
\left|
\det\left( \frac{\partial f_i}{\partial z_j}(a)
\right)_{\substack{ {1\le i \le m} \\ {1\le j \le m} }}
\right|
\ge (k_n)^m \prod_{j=1}^m r_j(a).
$$
Then define $V\in \C^{m+1}$ to be the vector of cofactors
$$
V_j:= (-1)^{m+1+j} \det \left( \frac{\partial f_i}{\partial z_l}(a)
\right)_{\substack{ {1\le i \le m} \\ {1\le l \le m+1, l\neq j} }}, 1\le j \le m+1.
$$
Choose $f_{m+1} \in \mathcal O(D,\D)$ to be extremal for the Carath\'eodory metric in the $V$
direction, so that $|f_{m+1}'(a)V|= C_D(a;V)$. By our choice of $V$,
$$
\det\left( \frac{\partial f_i}{\partial z_j}(a)
\right)_{\substack{ {1\le i \le m+1} \\ {1\le j \le m+1} }}
= \sum_{j=1}^{m+1} V_j \frac{\partial f_{m+1}}{\partial z_j}(a) = f_{m+1}'(a)V.
$$
By \eqref{1} and the recursion assumption,
\begin{multline*}
C_D(a;V) \ge k_n \sum_{j=1}^{m+1} |V_j|r_j(a) \ge k_n  |V_{m+1}|r_{m+1}(a)
\\
=  k_n  r_{m+1}(a) \left| \det\left( \frac{\partial f_i}{\partial z_j}(a)
\right)_{\substack{ {1\le i \le m} \\ {1\le j \le m} }}
 \right| \ge (k_n)^{m+1} \prod_{j=1}^{m+1} r_j(a).
\end{multline*}
\end{proof}

On the other hand, from the definition of the Carath\'eodory
metric, for any map $F\in\mathcal O(D,\D^n)$ one has that
$$|\det F'(a)|\le n!\Pi_D(a).$$

It follows that
\begin{equation}\label{23}
(k_n)^{2n}\le CE_D(a).(\Pi_D(a))^{-2}\le(n!)^2.
\end{equation}

Now we compare $\Pi_D(a)^{-2}$ with $VC_D(a)$. Define the diamond domain
$$
E_a:=\{z\in\C^n:\sum_{j=1}^nr_j(a)|z_j-a_j|<1\},
$$
the  inequalities \eqref{1} imply that $k_nD_a\subset E_a\subset D_a$, and so
\begin{equation}
\label{4}
k_n^{2n}VC_D(a)\le\mbox{Vol}(E_a)=\frac{(2\pi)^n}{(2n)!(\Pi_D(a))^2}\le VC_D(a).
\end{equation}
Combining \eqref{23} and \eqref{4}, we get that
$$
(k_n)^{2n}\le\frac{(2n)!}{(2\pi)^n}CE_D(a)VC_D(a)\le\frac{(n!)^2}{(k_n)^{2n}}.
$$
\end{proof*}

\noindent{\bf Remark.} Let $P_D(a)=1/\min\Pi_D(a),$ where the minimum is taken over all
orthonormal bases of $\C^n.$ Denote by $V_D^i(a)$ and $V_D^e(a)$ the maximal volume
of a polydisc $\Delta_a$ centered at $a$ such that $\Delta_a\subset D_a,$
respectively the minimal volume of a polydisc $\Delta_a$ centered at $a$ such that
$\Delta_a\supset D_a.$ It follows from the proof above that the functions
$V_D^i, V_D^e, VE_D, P_D$ and $(CE_D)^{-1}$ are equal up to small/large constants
depending only on $n.$

{}

\end{document}